\documentclass[11pt]{article}

\usepackage[centertags]{amsmath}
\usepackage{amsfonts}
\usepackage{amssymb}
\usepackage{amsthm}
\usepackage{newlfont}
\usepackage{bibspacing}
\usepackage{verbatim}

\newtheorem{thm}{Theorem}[section]
\newtheorem*{thm*}{Theorem}
\newtheorem{cor}[thm]{Corollary}

\newtheorem{prop}[thm]{Proposition}
\newtheorem*{prop*}{Proposition}

\newtheorem*{conj*}{Conjecture}
\newtheorem{defn}[thm]{Definition}
\newtheorem*{defn*}{Definition}
\theoremstyle{definition}
\newtheorem{rem}[thm]{\textbf{Remark}}
\newtheorem*{rmk*}{Remark}
\newtheorem*{fact*}{Fact}

\theoremstyle{proof}

\newcommand{\CD}{\text{CD}}

\newcommand{\II}{\text{II}}

\newcommand{\LogHess}{\text{LogHess}}
\newcommand{\Ric}{\text{Ric}}

\newcommand{\norm}[1]{\left\Vert#1\right\Vert}

\newcommand{\abs}[1]{\left\vert#1\right\vert}
\newcommand{\set}[1]{\left\{#1\right\}}
\newcommand{\brac}[1]{\left(#1\right)}
\newcommand{\scalar}[1]{\left \langle #1 \right \rangle}
\newcommand{\sscalar}[1]{\langle #1 \rangle}
\newcommand{\Real}{\mathbb{R}}

\newcommand{\eps}{\varepsilon}

\newlength{\defbaselineskip}
\newcommand{\setlinespacing}[1]           {\setlength{\baselineskip}{#1 \defbaselineskip}}

\numberwithin{equation}{section}

\begin{document}

\title{Harmonic Measures on the Sphere via Curvature-Dimension}
\author{Emanuel Milman\textsuperscript{1}\\ \\
\emph{Warmly dedicated to Dominique R. Bakry with great admiration}}
\date{}

\footnotetext[1]{Department of Mathematics,
Technion - Israel Institute of Technology, Haifa 32000, Israel. Supported by ISF (grant no. 900/10), BSF (grant no. 2010288) and Marie-Curie Actions (grant no. PCIG10-GA-2011-304066).
Email: emilman@tx.technion.ac.il.}

\maketitle

\begin{abstract}
We show that the family of probability measures on the $n$-dimensional unit sphere, having density proportional to:
\[
S^n \ni y \mapsto \frac{1}{\abs{y - x}^{n+\alpha}} ,
\]
satisfies the Curvature-Dimension condition $CD(n-1-\frac{n+\alpha}{4},-\alpha)$, for all $\abs{x} < 1$, $\alpha \geq -n$ and $n\geq 2$. The case $\alpha = 1$ corresponds to the hitting distribution of the sphere by Brownian motion started at $x$ (so-called ``harmonic measure" on the sphere). 
Applications involving isoperimetric, spectral-gap and concentration estimates, as well as potential extensions, are discussed. 
\end{abstract}

\section{Introduction}

In this note, we consider the following family of probability measures on the unit sphere $S^{n}$ in Euclidean space $(\Real^{n+1},\abs{\cdot})$:
\begin{equation} \label{eq:intro-family}
\mu_{x}^{n,\alpha} = \frac{c_x^{n,\alpha}}{\abs{y - x}^{n+\alpha}} d\sigma^n(y) ,
\end{equation}
where $n \geq 2$, $\abs{x} < 1$, $\alpha \in \Real$, $\sigma^n$ denotes the Haar probability measure on $S^n$ and $c_x^{n,\alpha} > 0$ is a normalization constant. In the case $\alpha=1$ we have $c_x^{n,1} = 1- \abs{x}^2$, and it is well known (e.g. \cite{SchechtmanSchmucky}) that $\mu_x^{n,1}$ is precisely the harmonic measure on $S^n$, characterized as the hitting distribution of $S^n$ by standard Brownian motion started at $x$, or equivalently, as the measure whose density is the Poisson kernel for the Laplace equation in $B^{n+1} = \set{x \in \Real^{n+1} ; \abs{x} < 1}$ with Dirichlet boundary conditions on $S^n = \partial B^{n+1}$. Denoting by $g$ the canonical Riemannian metric on $S^n$, the triplet $(S^n,g,\mu_{x}^{n,\alpha})$ constitutes a weighted Riemannian manifold. 

Clearly, given $n$, $\alpha$ and $\abs{x}$, the density of such measures on $y \in S^n$ only depends on the angle $\theta(y) \in [0,\pi]$ given by
$\cos(\theta(y)) = \sscalar{y, x/\abs{x}}$,
and so various functional and concentration properties of these measures may be reduced to the study of the one-dimensional measure $\theta_*(\mu_x^{n,\alpha})$, whose density on $[0,\pi]$ is proportional to:
\begin{equation} \label{eq:1d-density}
\frac{\sin^{n-1}(\theta) }{ ( 1 - 2 \cos(\theta) \abs{x} + \abs{x}^2)^{\frac{n+\alpha}{2}}} .
\end{equation}
This is essentially the approach taken by G. Schechtman and M. Schmuckenschl\"{a}ger \cite{SchechtmanSchmucky} and F. Barthe, Y. Ma and Z. Zhang \cite{BartheMaZhang} in their study of the harmonic case $\alpha = 1$ (further details to be provided momentarily). 

While such an approach will certainly yield the most precise results, it is very specialized to the case of having a weighted manifold invariant under a huge symmetry group. Furthermore, the calculations involving the precise density (\ref{eq:1d-density}) may be at times tedious, or alternatively rely on the Brownian motion interpretation of the harmonic case $\alpha = 1$. In this note we would like to take a different path, which while yielding somewhat less precise estimates, is essentially effortless, and may be applied in far greater generality. Our main observation is the following:
  
\begin{thm} \label{thm:intro-main}
For all $n \geq 2$, $\abs{x} < 1$ and $\alpha \geq -n$, the weighted manifold $(S^{n},g,\mu_x^{n,\alpha})$ satisfies the Curvature-Dimension condition 
$CD(n-1 - \frac{n+\alpha}{4}, - \alpha)$. 
\end{thm}

We defer the definition of the Bakry--\'Emery Curvature-Dimension $CD(\rho,N)$ condition to Section \ref{sec:CD}, and presently only remark that $\rho \in \Real$ represents a lower bound on the generalized Ricci curvature of the given weighted manifold, whereas $N \in \Real$ represents an upper bound on its generalized dimension. Indeed, when $\alpha = -n$, corresponding (for all $x$) to the case of the canonical $n$-sphere equipped with its uniform Haar measure, our estimate becomes $CD(n-1,n)$, in precise agreement with the $n$-sphere's usual constant Ricci curvature of $(n-1) g$. 
Let us also remark that the Curvature-Dimension definition we employ only coincides with the original Bakry--\'Emery definition \cite{BakryEmery} in the range $N \in (-\infty,0) \cup [n,\infty]$, but not for $N \in [0,n)$ - see Section \ref{sec:CD} for more information. 

\medskip

Properties of weighted manifolds (and more generally, Markov diffusion processes) satisfying the $CD(\rho,N)$ condition have been intensively studied in the past three decades, and in the last decade the scope has been extended to include very general geodesic metric-measure spaces satisfying a synthetic version of the $CD(\rho,N)$ condition - see e.g. \cite{BakryEmery,BakryStFlour,QianWeightedVolumeThms,LedouxLectureNotesOnDiffusion,Ledoux-Book,CMSInventiones, CMSManifoldWithDensity, BakryQianGenRicComparisonThms, PerelmanEntropyFormulaForRicciFlow, VonRenesseSturmRicciChar, SturmCD12, LottVillaniGeneralizedRicci, WeiWylie-GenRicciTensor, MorganBook4Ed, BGL-Book,Wang-ManifoldsBook, EMilmanSharpIsopInqsForCDD, KolesnikovEMilman-Reilly} and the references therein for a very non-comprehensive exert. 
In the study of $n$-dimensional weighted-manifolds satisfying $CD(\rho,N)$, the range of admissible values for $N$ has traditionally been $N \in [n,\infty]$. However, returning to the family of measures (\ref{eq:intro-family}), note that this traditional range excludes the harmonic case $N = -\alpha = -1$. 
Fortunately, in recent years, this range has been extended to also include $N \in (-\infty,1)$ - see e.g. \cite{EMilmanNegativeDimension} for a recent account. 

While in the weighted manifold setting this extended range constitutes a recent development, in the Euclidean setting with $\rho = 0$ and $N \in (-\infty,0) \cup [n,\infty]$, the class of $CD(0,N)$ spaces $(\Real^n,\abs{\cdot},\mu)$ coincides with the class of convex measures (of full-dimensional convex support and $C^2$ density $\Psi$), introduced by Borell \cite{Borell-logconcave} in the 70's (cf. Brascamp--Lieb \cite{BrascampLiebPLandLambda1}) and studied by S. Bobkov and M. Ledoux \cite{BobkovConvexHeavyTailedMeasures,BobkovLedouxWeightedPoincareForHeavyTails}. In this latter setting, the case $N \in (-\infty,0)$ corresponds to ``heavy-tailed" measures, characterized by the requirement that $1/\Psi^{1/(n-N)}$ be convex. A prototypical example is the Cauchy probability measure on $(\Real^n,\abs{\cdot})$:
\[
\nu^{n,\alpha} = \frac{c_{n,\alpha}}{ (1+ \abs{y}^2 )^{\frac{n+\alpha}{2}}} dy \;\; , \;\; \alpha > 0 , 
\]
which satisfies $CD(0,-\alpha)$. Note that our measures (\ref{eq:intro-family}) may be thought of as spherical analogues of these Euclidean Cauchy measures, and so in hindsight it is not so surprising that they also satisfy a Curvature-Dimension condition with negative dimension (albeit with positive curvature contributed by the sphere).

\medskip

While Theorem \ref{thm:intro-main} boils down to an elementary calculation, our main incentive for writing this note is to demonstrate and advocate the usefulness of the $CD(\rho,N)$ condition in the less established range $N \in (-\infty,1)$. 
As immediate corollaries of Theorem \ref{thm:intro-main}, we apply the recently available isoperimetric, spectral and concentration results for that range. 
The case $N = -\alpha \in [1,n)$ is excluded from the ensuing discussion, since in general none of the above good properties hold in that range (as explained in \cite{EMilmanNegativeDimension}).

\medskip

We henceforth assume that $\alpha \in (-1, 3n-4)$, and set:
\[
\rho_{n,\alpha} := n-1 - \frac{n+\alpha}{4} > 0 . 
\]
First, we apply the isoperimetric comparison theorems obtained in our previous works \cite[Corollary 1.4, Theorem 6.1]{EMilmanNegativeDimension}, to deduce the strongest form of information regarding the measures (\ref{eq:intro-family}). Given a metric space $(\Omega,d)$ endowed with a Borel measure $\mu$, recall that the Minkowski (exterior) boundary measure $\mu^+(A)$ of a Borel set $A \subset \Omega$ is defined as $\mu^+(A) := \liminf_{\eps \to 0} \frac{\mu(A^d_{\eps} \setminus A)}{\eps}$, where $A_{\eps}=A^d_{\eps} := \set{x \in \Omega ; \exists y
\in A \;\; d(x,y) < \eps}$ denotes the $\eps$ extension of $A$ with
respect to the metric $d$. In our context, the metric $d$ is given by the standard geodesic distance on $(S^n,g)$.

\begin{cor} \label{cor:isop}
Given $n \geq 2$ and $\alpha \in (-1, 3n-4)$, set:
\[
\delta_{n,\alpha} := \frac{\rho_{n,\alpha}}{\alpha+1} > 0, 
\]
and denote the following functions on $\Real$:
\[
\varphi_{n,\alpha}(t) := \frac{c_{n,\alpha}}{\cosh(\sqrt{\delta_{n,\alpha}} t)^{\alpha+1}} \; \; , \;\; \Phi_{n,\alpha}(t) := \int_{-\infty}^t \varphi(s) ds ,
\]
where $c_{n,\alpha} > 0$ is a normalization constant to make $\varphi_{n,\alpha}$ a probability density. 
Then for any $\abs{x} < 1$ and Borel set $A \subset S^{n}$, the following isoperimetric inequality holds:
\[
(\mu_{x}^{n,\alpha})^+(A) \geq \varphi_{n,\alpha}\circ \Phi_{n,\alpha}^{-1} (\mu_{x}^{n,\alpha}(A)) .
\]
In particular, the following Cheeger-type isoperimetric inequality holds:
\[
(\mu_{x}^{n,\alpha})^+(A) \geq D_{Che}^{n,\alpha} \min(\mu_{x}^{n,\alpha}(A),1-\mu_{x}^{n,\alpha}(A)) 
\]
with:
\[
D_{Che}^{n,\alpha} \geq \sqrt{\delta_{n,\alpha}} \frac{1}{\int_0^\infty \cosh(t)^{1+\alpha} dt}  \geq c_0 \sqrt{\rho_{n,\alpha}} \min(1 , \sqrt{1 + \alpha}) 
 \]
 and $c_0 > 0$ a numeric constant. 
\end{cor}

In fact, a slightly more refined version of the above result may be obtained by incorporating information on the diameter of the sphere - see \cite{EMilmanNegativeDimension}. 
As an immediate corollary, we obtain the following two-level concentration estimate \cite[Proposition 6.4]{EMilmanNegativeDimension}:
\begin{cor}
With the above assumptions and notation, for any Borel set $A \subset S^n$ with $\mu_{x}^{n,\alpha}(A) \geq 1/2$, we have for any $r \in (0,\pi)$:
\[ \mu_{x}^{n,\alpha}(S^n \setminus A_r) \leq \int_r^\infty \varphi_{n,\alpha}(t) dt \leq \begin{cases}   
C \min(1,\sqrt{\alpha+1}) \frac{\exp(- c \rho_{n,\alpha} r^2)}{1 + \sqrt{\rho_{n,\alpha}} r }  & r \in \left [ 0,\sqrt{\frac{\alpha+1}{\rho_{n,\alpha}}} \right ] \\ C \min(1, \frac{1}{\sqrt{\alpha+1}}) \exp(- c \sqrt{\alpha+1} \sqrt{\rho_{n,\alpha}} r) & \text{otherwise} 
\end{cases} ,
\] where $c , C > 0$ are numeric constants.
\end{cor}

Note that in the harmonic case $\alpha = 1$, our concentration estimates only yield exponential tail decay, and are thus inferior to the uniform sub-Gaussian estimate:
\begin{equation} \label{eq:sub-Gaussian}
\mu_{x}^{n,\alpha}(S^n \setminus A_r) \leq C \exp(-c n r^2) ,
\end{equation}
obtained by Schechtman--Schmuckenschl\"{a}ger \cite{SchechtmanSchmucky}. However, as soon as $a n \leq \alpha \leq b n$ for some $0< a < b < 3$, observe that our concentration estimates do in fact become sub-Gaussian of the form (\ref{eq:sub-Gaussian}), with $c = c(a,b)$ and $C = C(a,b)$. 

\medskip
Finally, we apply the Lichnerowicz spectral-gap estimate from our previous work with A. Kolesnikov \cite{KolesnikovEMilman-Reilly} (proved for weighted manifolds with convex boundaries, and also independently obtained by S.-I. Ohta \cite{Ohta-NegativeN} in the case of closed manifolds); the improvement in the range $\alpha \in (-1,1)$ below is obtained by invoking the Maz'ya--Cheeger inequality as in \cite[Theorem 6.1]{EMilmanNegativeDimension}. Let $\lambda_x^{n,\alpha}>0$ denote the spectral-gap of $(S^n,g,\mu_x^{n,\alpha})$, i.e. the maximal constant $\lambda > 0$ so that:
\begin{equation} \label{eq:Poincare}
\int_{S^n} g(\nabla f, \nabla f) d\mu_x^{n,\alpha} \geq \lambda \int_{S^n} f^2  d\mu_x^{n,\alpha},
\end{equation}
for all smooth functions $f : (S^n,g) \rightarrow \Real$ with $\int_{S^n} f d\mu_x^{n,\alpha} = 0$. 
\begin{cor}
For all $\abs{x} < 1$ and $\alpha \in (0, 3n-4)$, we have:
\[
\lambda_x^{n,\alpha} \geq \frac{\alpha}{\alpha+1} \rho_{n,\alpha} .
\]
When $\alpha \in (-1,1)$, we also have:
\[
\lambda_x^{n,\alpha} \geq \frac{c^2_0}{4}  \min(1 , 1 + \alpha) \rho_{n,\alpha} . 
\]
\end{cor}
In the harmonic case $\alpha= 1$, this should be compared with the estimate:
\begin{equation} \label{eq:gap1}
 n \geq \lambda_x^{n,1} \geq \frac{n-1}{2} 
\end{equation}
obtained by Barthe--Ma--Zhang in \cite{BartheMaZhang}. While our estimate yields the inferior bound:
\begin{equation} \label{eq:gap2}
\lambda_x^{n,1} \geq \frac{3}{8} (n-1) - \frac{1}{4}
\end{equation}
it is nevertheless of the correct order. Barthe--Ma--Zhang also showed that no log-Sobolev inequality which holds uniformly for all $\abs{x}< 1$ is possible on $(S^n,g,\mu_x^{n,\alpha})$ when $\alpha = 1$,  nevertheless obtaining an essentially sharp estimate depending on $\abs{x}$. While we do not pursue a similar direction here, we comment that this indeed agrees with the model space $(\Real,\abs{\cdot},\varphi_{n,\alpha}(t) dt)$ from Corollary \ref{cor:isop}, which does not satisfy any log-Sobolev inequality when $N = -\alpha \in (-\infty,1)$. Note that in general, the spectral-gap estimates (\ref{eq:gap1}) or (\ref{eq:gap2}) do not yield the sub-Gaussian concentration (\ref{eq:sub-Gaussian}), but only exponential concentration \cite{GromovMilmanLevyFamilies}. It is the $CD(\rho,N)$ condition for $\rho > 0$ and $N \in (-\infty,1)$ which precisely reconciles between spectral-gap, lack of log-Sobolev inequality, and sub-Gaussian concentration in the range $r \in (0,\sqrt{(1-N)/\rho})$ (see \cite{EMilmanNegativeDimension}).

\medskip

The rest of this work is organized as follows. In Section \ref{sec:CD} we recall the definition of the Curvature-Dimension condition. In Section \ref{sec:sphere} we provide a proof of Theorem \ref{thm:intro-main}. In Section \ref{sec:conclude} we provide some concluding remarks; in particular, we comment on how to extend the class (\ref{eq:intro-family}) to measures involving norms more general than Euclidean (and without assuming any symmetry). 

\medskip \noindent
\textbf{Acknowledgements.} It is a pleasure to thank Dominique Bakry and Michel Ledoux for their various (already classical) text-books on Concentration of Measure and related areas - they serve as a never-ending source of ideas, knowledge and inspiration.

\section{Curvature-Dimension Condition} \label{sec:CD}
 
 Let $(M^n,g)$ denote an $n$-dimensional ($n \geq 2$) complete connected oriented smooth Riemannian manifold, and let $\mu$ denote a measure on $M$ having density $\Psi$ with respect to the Riemannian volume form $vol_g$. For simplicity, we assume that $M$ is closed, i.e. compact and without boundary, but all the results we will use equally apply when $M$ is only assumed geodesically convex. We assume that $\Psi$ is positive and $C^2$ smooth on the entire manifold. As usual, we denote by $Ric_g$ the usual Ricci curvature tensor and by $\nabla_g$ the Levi-Civita covariant derivative.

\begin{defn*}[Generalized Ricci Tensor]
Given $N \in (-\infty,\infty]$, the $N$-dimensional generalized Ricci curvature tensor $Rig_{g,\mu,N}$ is defined as:
\begin{equation} \label{eq:Ric-Tensor}
Ric_{g,\mu,N}  :=  Ric_g - \LogHess_{N-n} \Psi ,
\end{equation}
where:
\[
\LogHess_{N-n} \Psi := \nabla^2_g \log(\Psi) + \frac{1}{N-n} \nabla_g \log(\Psi) \otimes \nabla_g \log(\Psi) =  (N-n) \frac{\nabla^2_g \Psi^{\frac{1}{N-n}}}{\Psi^{\frac{1}{N-n}}} .
\]
To make sense of the latter tensor when $N-n \in \set{0,\infty}$, we employ throughout the convention that $\frac{1}{\infty} = 0$, $\frac{1}{0} = +\infty$ and $\infty \cdot 0 = 0$. 
\end{defn*}

\begin{defn}[Curvature-Dimension condition] \label{def:CD}
$(M^n,g,\mu)$ satisfies the Curvature-Dimension condition $\CD(\rho,N)$ ($\rho \in \Real$ and $N \in (-\infty,\infty]$) if $\Ric_{g,\mu,N} \geq \rho g$ as symmetric $2$-tensors on $M$. 
\end{defn}

Note that $\CD(\rho,N)$ is satisfied with $N=n$ if and only if $\Psi$ is constant and $Ric_{g,\mu,n} = Ric_g \geq \rho g$ (the classical constant density case). The generalized Ricci tensor (\ref{eq:Ric-Tensor}) was introduced with $N=\infty$ by Lichnerowicz \cite{Lichnerowicz1970GenRicciTensorCRAS,Lichnerowicz1970GenRicciTensor} and in general by Bakry \cite{BakryStFlour} (cf. Lott \cite{LottRicciTensorProperties}). The Curvature-Dimension condition was introduced by Bakry and \'Emery for $N \in [n,\infty]$ in equivalent form using $\Gamma$-calculus in \cite{BakryEmery} (in the more abstract framework of diffusion generators). Its name stems from the fact that the generalized Ricci tensor incorporates information on curvature and dimension from both the geometry of $(M,g)$ and the measure $\mu$, and so $\rho$ may be thought of as a generalized-curvature lower bound, and $N$ as a generalized-dimension upper bound.

Let us give a bit more background on the original definition given by Bakry and \'Emery in \cite{BakryEmery}. 
Given a generator $L$ of a Markov diffusion process (see \cite{BGL-Book} for more details), Bakry--\'Emery defined the $CD(\rho,N)$ condition as the requirement that:
\[
\Gamma_2(f) \geq \rho \Gamma(f) + \frac{1}{N} (Lf)^2 ,
\]
for all appropriate test functions $f$; in our weighted-manifold setup, $\Gamma(f) = g(\nabla f,\nabla f)$, $\Gamma_2(f) = (\Ric_g - \nabla_g^2 \log \Psi)(\nabla f,\nabla f)$ and $L = \Delta_g + \nabla_g \log \Psi$ denotes the generalized Laplacian. 
 The equivalence of the above two definitions was established by Bakry \cite{BakryStFlour} for the traditional range $N \in [n,\infty]$ and extended to the range $N \in (-\infty,0)$ in \cite{KolesnikovEMilman-Reilly}. However, as observed in \cite[Section 7]{EMilmanNegativeDimension}, the two definitions are no longer equivalent in the range $N \in [0,n)$, which includes the case $N \in [0,1)$ relevant to the results in this work. Consequently, we emphasize that we will be using the tensorial Definition \ref{def:CD}. In particular, note that our version of the Curvature-Dimension condition $CD(\rho,N)$ is clearly monotone in $\rho$ and in $\frac{1}{N-n}$.

\section{Calculation on the Sphere} \label{sec:sphere}

In this work, we specialize to the canonical $n$-sphere $(S^n,g)$ endowed with the probability measure $\mu^{n,\alpha}_x$ given in (\ref{eq:intro-family}). Let us denote its density with respect to the Haar probability measure $d\sigma_n(y)$ by $\Psi^{n,\alpha}_x(y)$, i.e.:
\[
\Psi^{n,\alpha}_x(y) := \frac{c_x^{n,\alpha}}{\abs{y - x}^{\alpha}} .
\]
Recall that the classical Ricci tensor of the canonical $n$-sphere satisfies $Ric_g = (n-1) g$. Consequently, Theorem \ref{thm:intro-main} will follow once we establish that for any $n \geq 2$, $\alpha \geq -n$, $\abs{x} < 1$, setting:
\[
N = - \alpha ,
\]
we have:
\begin{equation} \label{eq:LogHess-Spec}
-\LogHess_{N-n} \Psi^{n,\alpha}_x = -(N-n)  \frac{\nabla^2_g (\Psi_x^{n,\alpha})^{\frac{1}{N-n}}}{(\Psi_x^{n,\alpha})^{\frac{1}{N-n}}}\geq - \frac{n+\alpha}{4}  g .
\end{equation}
Since $-(N-n) = n+\alpha \geq 0$ in the above range, Theorem \ref{thm:intro-main} boils down to showing the following:

\begin{prop}
For every $\abs{x} < 1$: \[
\frac{\nabla_{S^n}^2 \abs{\cdot-x}}{\abs{\cdot-x}} \geq - \frac{1}{4} g , \]
where $\nabla_{S^{n}} = \nabla_g$ denotes the covariant derivative on the canonical $n$-sphere $(S^{n},g)$.
\end{prop}
\begin{proof}
First, recall (e.g. \cite{GHLBookEdition3}) that for any $C^2$ function $f$ defined on a neighborhood of $S^{n}$ in $\Real^{n+1}$:
\[
\nabla_{S^n}^2 f = \nabla^2_{\Real^{n+1}} f - f_\nu \II^\nu_{S^n} , 
\]
where $\nabla = \nabla_{\Real^{n+1}}$ denotes the standard covariant derivative in Euclidean space $(\Real^{n+1},\scalar{\cdot,\cdot})$, $\nu$ is the outward unit normal to $S^n$ in its standard embedding in $\Real^{n+1}$, $f_\nu$ is the derivative of $f$ in the direction of $\nu$, and $\II = \II_{S^n}$ denotes the second-fundamental form of the latter embedding with respect to $\nu$, i.e. $\II(X,Y) = \scalar{\nabla_X \nu, Y}$ (note our slightly non-standard convention for the direction of the normal). Trivially, in our case we have $\II^\nu_{S^n} = g$. 

Now fix $y \in S^{n}$ and a unit vector $\theta$ in the tangent space $T_{y} S^{n}$. We identify all tangent spaces with the corresponding subspaces of $\Real^{n+1}$, so that $\scalar{y,\theta} = 0$. Our task is then to show that:
\begin{equation} \label{eq:task1}
 \frac{\scalar{\nabla^2 f (y) \theta, \theta} - \scalar{\nabla f (y) , y}}{f(y)} \geq -\frac{1}{4} ,
\end{equation}
for $f(y) = \abs{y-x}$ and all $y$ and $\theta$ as above, where all differentiation is now with respect to the standard connection on $\Real^{n+1}$. Calculating:
\[
 \nabla f(y) = \frac{y-x}{\abs{y-x}} \;\; , \;\; \nabla^2 f (y) = \frac{1}{\abs{y-x}} \brac{ I_{\Real^{n+1}} - \frac{y-x}{\abs{y-x}} \otimes \frac{y-x}{\abs{y-x}} } ,
\]
and so (\ref{eq:task1}) boils down to showing:
\[
 \frac{1}{\abs{y-x}}  \brac{1 - \frac{\scalar{y-x,\theta}^2}{\abs{y-x}^2} - \scalar{y-x,y}} = \frac{1}{\abs{y-x}^2} \brac{\scalar{x,y} - \frac{\scalar{x,\theta}^2}{\abs{y-x}^2}} \geq -\frac{1}{4} .
\]

To this end, denote:
\[
a = \scalar{x,y} \;\;, \;\; b = \scalar{x,\theta},
\]
and note that:
\[
 \abs{x}^2 = a^2 + b^2 \;\; , \;\; \abs{y-x}^2 = (1-a)^2 + b^2 .
\]
Consequently, we see that our problem reduces to calculating the minimum of the function:
\begin{equation} \label{eq:F}
F(a,b) := \frac{a (( 1-a)^2 + b^2) - b^2}{(( 1-a)^2 + b^2)^2} ,
\end{equation}
on the circle $\set{a^2 + b^2 \leq 1}$. Denoting $d := ( 1-a)^2 + b^2$, we have:
\[
F(a,b) = \frac{ -(1-a) d + (1-a)^2 }{d^2} \geq \frac{ - \frac{1}{4} d^2}{d^2} = -\frac{1}{4} ,
\]
and indeed this minimal value is attained on the circle's entire boundary $\set{a^2 + b^2 = 1}$ where $1-a = d/2$. This concludes the proof.

\end{proof}

\begin{rem}
A-priori there is no reason to use $N-n = -(n+\alpha)$ in the above proof, and we could have also proceeded with $(N-n) p = -(n+\alpha)$ for some parameter $p > 0$. 
Repeating the above argument, everything boils down to calculating the minimum of:
\begin{equation} \label{eq:Fp}
 F_p(a,b) := \frac{a (( 1-a)^2 + b^2) + (p-2) b^2}{(( 1-a)^2 + b^2)^2} 
\end{equation}
on the unit circle. It turns out that when $p \geq 1$, this minimum is still $-\frac{1}{4}$ (attained at $(a,b) = (-1,0)$), yielding:
\[
\frac{\nabla_{S^n}^2 \abs{\cdot-x}^p}{\abs{\cdot-x}^p} \geq -\frac{p}{4} g .
\]
This would have resulted in an inferior Curvature-Dimension condition $CD(\rho,N)$, with the same curvature $\rho = n-1 - \frac{n+\alpha}{4}$ but with a worse dimension $N = n - \frac{n+\alpha}{p}$. On the other hand, when $p \in (0,1)$, $F_p$ is not bounded below on the unit circle, as seen by setting $b^2 = 1 - a^2$ and letting $a \rightarrow 1$. Consequently, it seems we cannot gain from playing such a game, at least if our goal is to obtain uniform estimates in $\abs{x} < 1$. 
\end{rem}

\begin{rem}
By using $b=0$ in (\ref{eq:F}) or (\ref{eq:Fp}) and letting $a \rightarrow 1$, one sees that $F(a,b)$ is not bounded from above in $\set{a^2 + b^2 < 1}$. This prevents us from applying our results in the range $\alpha < -n$, when $N-n > 0$ in (\ref{eq:LogHess-Spec}).
\end{rem}

\section{Concluding Remarks}  \label{sec:conclude}
 
\subsection{Better Understanding of Harmonic Case}

It is unfortunate that the harmonic case $\alpha=1$ does not satisfy $CD(an,-bn)$ for some constant $a\in (0,1)$ and $b > 0$; as explained in the Introduction, this would have recovered the Schechtman--Schmuckenschl\"{a}ger sub-Gaussian concentration estimate (\ref{eq:sub-Gaussian}). It therefore seems that our current understanding is missing some additional subtle curvature property of the harmonic measure. Perhaps a finer analysis using the forthcoming Graded Curvature-Dimension condition \cite{EMilman-GradedCD} would enable resolving this shortcoming.

\subsection{Beyond $\alpha \in (-1,3n-4)$}
 
We do not know what happens beyond this range for $\alpha$, and moreover, we do not have any clear intuition regarding what to expect. 

It could be that the negative curvature of the measure overcomes the positive curvature of the sphere around $\alpha \simeq 3n$, resulting in a dimensional degradation or perhaps even phase-transition in the behaviour of $\alpha \mapsto \inf_{\abs{x} < 1} \lambda_x^{n,\alpha}$; this would give credence to the (surprising?) $-\frac{1}{4}$ constant appearing in Theorem \ref{thm:intro-main}. On the other hand, this could be an artifact of our proof. 

In the vicinity of $\alpha = -1$ all of our estimates breakdown - we tend to believe that this is an artifact of our (elementary) proof. Just by interpolating between the cases $\alpha = -n$ and $\alpha = 0$, we believe that the spectral-gap remains of the order of $n$ uniformly in $\alpha \in [-n,0]$ and $\abs{x} < 1$. 

And as for $\alpha < -n$ -  we do not know what to expect. 
 
\subsection{Extension to more general norms}

Recall that the proof of Theorem \ref{thm:intro-main} boiled down to estimating from below:
\[
\frac{\nabla_{S^n}^2 \abs{\cdot - x}}{\abs{\cdot-x}} . 
\]
Let us check what happens if we replace the Euclidean norm $\abs{\cdot}$ by a more general (say $C^2$ smooth) one $\norm{\cdot}$ - this is interesting enough even for $x=0$. We would like to have:
\[
\frac{\scalar{\nabla_{S^n}^2 \norm{y} \theta,\theta}}{\norm{y}} = \frac{\scalar{ \nabla^2_{\Real^{n+1}} \norm{y} \theta , \theta} - \scalar{\nabla_{\Real^{n+1}} \norm{y} , y}}{\norm{y}} \geq -(1-\eps) ,
\]
for all $y,\theta \in S^{n}$ with $\scalar{y,\theta} = 0$. But since $\scalar{\nabla_{\Real^{n+1}} \norm{y} , y} = \norm{y}$, we see that the desired condition is that:
\begin{equation} \label{eq:2convex}
\exists \eps > 0 \;\;\; \scalar{ \nabla^2_{\Real^{n+1}}\norm{y} \theta , \theta} \geq \eps \norm{y} \;\;\; \forall y,\theta \in S^{n} \;\; \scalar{y,\theta} = 0 . 
\end{equation}
In such a case, denoting:
\[
\eta^{n,\alpha} := \frac{c_{n,\alpha}}{\norm{y}^{n+\alpha}} d\sigma^n(y) ,
\]
we would have that $(S^n,g,\eta^{n,\alpha})$ satisfies $CD(n-1-(1-\eps)(n+\alpha) , -\alpha)$ for all $\alpha \geq -n$, implying good concentration properties when $\alpha \in (-1 , \frac{n-1}{1-\eps} - n)$. The condition (\ref{eq:2convex}) is reminiscent of the condition for the norm to have a quadratic modulus of convexity \cite{LT-Book-II}, but is not invariant under linear transformations.

\setlinespacing{0.82}
\setlength{\bibspacing}{2pt}

\bibliographystyle{plain}
\bibliography{../ConvexBib}

\end{document}